%% file: main.tex
\documentclass[a4paper,12pt,leqno]{amsart}
\usepackage[colorlinks=true,linkcolor=darkblue,citecolor=darkblue]{hyperref}
\input{Preamble}
\definecolor{darkblue}{rgb}{0.0, 0.0, 0.55}
\usepackage{verbatim}

\begin{document}

\title[]{Balanced groups and the virtually cyclic dimension of poly-surfaces groups}

\author{Jesús Hernández Hernández}
\address{Centro de Ciencias Matemáticas, Universidad Nacional Autónoma de México. Morelia, Michoacán, México 58089}
\email{jhdez@matmor.unam.mx}

\address{Instituto de Matemáticas, Universidad Nacional Autónoma de México. Oaxaca de Juárez, Oaxaca, México 68000}

\email{porfirio.leon@im.unam.mx}

\author[Porfirio L. León Álvarez ]{Porfirio L. León Álvarez}



\date{\today}



\begin{abstract}


In this paper we obtain explicit linear upper bounds for the virtually cyclic dimension of normally poly-surface and normally poly-free groups. Our approach is based on a structural study of the balanced property (L\"uck’s Condition~C), which provides structural control over commensurators of virtually cyclic subgroups.

We prove general stability results showing that the balanced property is preserved under suitable short exact sequences, direct limits, and acylindrical graph of groups decompositions. As applications, we establish that normally poly-hyperbolic groups, normally poly-free groups, and normally poly-surface groups are balanced. These classes include, in particular, pure braid groups of surfaces with non-empty boundary, Artin groups of FC-type, right-angled Artin groups, and fundamental groups of mapping tori of surface homeomorphisms.



\end{abstract}

\maketitle
\setcounter{tocdepth}{1}
\tableofcontents

\section{Introduction}

Let $G$ be a discrete group and let $\mathcal{F}$ be a collection of subgroups of $G$. We say that $\mathcal{F}$ is a \emph{family of subgroups} if it is nonempty and closed under taking subgroups and under conjugation in $G$. A \emph{model for the classifying space} $E_{\mathcal{F}}G$ is a $G$-CW-complex $X$ such that all its isotropy groups belong to $\mathcal{F}$ and it satisfies the following universal property: for any $G$-CW-complex $Y$ whose isotropy groups lie in $\mathcal{F}$, there exists a $G$-map $Y \to X$ that is unique up to $G$-homotopy. It is a standard result that such a model always exists. For more details on classifying spaces for families of subgroups  we refer the reader to \cite{Lu05}.

The \emph{virtually cyclic dimension} of a discrete group $G$, denoted by $\gdvc(G)$, is defined as the minimal integer $n$ for which there exists an $n$-dimensional model for the classifying space $E_{\vcyc}G$, where $\vcyc$ denotes the family of virtually cyclic subgroups of $G$.  These spaces are of particular importance since they appear in the formulation
of the Farrell--Jones conjecture (see \cite{LR05}), which has attracted considerable
attention in geometric group theory and motivated explicit computations of the
virtually cyclic dimension for several prominent classes of groups, such as
mapping class groups (see \cite{Nucinkis:Petrosyan} \cite{JPT16}), 3-manifold groups (see \cite{MR4303333}), virtually poly-$\Z$ groups (see \cite[Theorem 5.13]{LW12}), and classical Artin braid groups (see \cite{Ramon:Juan}). In this paper we give linear upper bounds for the virtually cyclic dimension of normally poly-surfaces and normally poly-free groups.

Now, let $\mathcal{P}$ be a class of groups. We say that $G$ is \emph{poly-$\mathcal{P}$} if there exists a filtration $$1 = G_0 \lhd G_1 \lhd \cdots \lhd G_{n-1} \lhd G_n = G,$$ such that each quotient $G_{i+1}/G_i$ belongs to $\mathcal{P}$. If additionally each $G_i$ is also normal in $G$, we say that $G$ is \emph{normally poly-$\mathcal{P}$}. The minimal $n$ for which such a filtration exists is called the \emph{length} of $G$.

In particular, if $\mathcal{P}$ is either the class of free groups or the class of fundamental groups of connected surfaces (possibly non-orientable, possibly of infinite type) with non-positive Euler characteristic and compact (if non-empty) boundary, then we say that $G$ is either normally poly-free or normally poly-surface, respectively. Typical examples of normally poly-free groups include pure braid groups of surfaces with non-empty boundary (see \cite{MR1797585} \cite[Theorem 1.5]{arXiv250610706L}), even Artin groups of FC-type (see \cite[Theorem 3.18]{MR3900773}), (see \cite[Theorem A]{MR4360030}), and right-angled Artin groups. Note that it has been proved that normally poly-free groups satisfy the Farrell--Jones conjecture (see \cite[Theorem A]{MR4246787}), see also \cite{MR1797585},  \cite[Theorem 1.1]{MR4565703}, \cite[Theorem 2.3.7]{JPSS}

In this work, we obtain explicit upper bounds for the virtually cyclic dimension of these types of groups.

\begin{theorem}\label{poli:surface:th1}
Let $G$ be a normally poly-surface group of length $n$, then $G$ is balanced and its virtually cyclic dimension satisfies $\gdvc(G)\leq 2n.$ 
\end{theorem}

\begin{theorem}\label{poli:free:virtually:cy}
Let $G$ be a normally poly-free group of length $n$, then $G$ is balanced and its virtually cyclic dimension satisfies  $\gdvc(G)\leq n+1.$ 
\end{theorem}

\begin{remark}\ \begin{itemize}
    \item \cref{poli:free:virtually:cy} improves \cite[Theorem~1.1]{MR4817678}, where a linear upper bound is also established. Note also that the upper bound in \cref{poli:free:virtually:cy} is sharp: $\Z^n$ for all $n\ge2$, realizes the equality.
    \item Some poly-$\Z$ groups, classical pure Artin braid groups, and right-angled Artin groups (RAAGs) are examples of normally poly-free groups. Their virtually cyclic dimension has been computed in \cite[Theorem 5.13]{LW12}, \cite{Ramon:Juan}, and \cite[Theorem 1.4]{MR4818170}, respectively.

    \end{itemize}
\end{remark}

To prove \cref{poli:surface:th1} and \cref{poli:free:virtually:cy}, we first prove that normally poly-surface groups are balanced (see \cref{hyperbolic:groups:propertyC} and \cref{rem:poly-free and poly-surface are balanced}). Once this is established, we apply \cref{luck:condition C} to reduce the problem to the computation of the proper geometric dimension of the Weyl groups of infinite cyclic subgroups.

Let \( G \) be a group. We say that \( G \) is \emph{balanced}, or that it satisfies \emph{Lück’s Condition~C}, if for every element \( h \in G \) of infinite order and for all nonzero integers \( m,n \), the fact that the powers \( h^{m} \) and \( h^{n} \) are conjugate in \( G \) necessarily implies that \( |m| = |n| \).

Examples of balanced groups include hyperbolic groups (see \cite{BH99}), mapping class groups (see \cite[Proposition 4.1]{JPT16}), outer automorphism groups of finitely generated free groups \cite[Proposition 3.1]{2023arXiv230801590G}, $\cat(0)$ groups (see \cite[Proof of Theorem 0.1]{MR2545612}).

The property of being balanced has several applications in geometric group theory and algebraic topology. 
For instance, if a group \( G \) is balanced, then the commensurator of an infinite cyclic subgroup can be realized as a normalizer. 
More precisely, let \( H \) be an infinite cyclic subgroup of \( G \). We define the \emph{commensurator} and the \emph{normalizer} of \( H \) in \( G \) as follows:
\[
N_G[H] := \{ g \in G \mid gHg^{-1} \text{ is commensurable with } H \}, \qquad
N_G(H) := \{ g \in G \mid gHg^{-1} = H \}.
\] 
If \( G \) is balanced and either \( N_G[H] \) is finitely generated or \( G \) satisfies the unique root property, then there exists a subgroup \( H' \leq G \) commensurable with \( H \) such that \( N_G[H] = N_G(H') \) (see \cref{conditionC:conmen:fg} and \cite[Proposition 3.4]{MR3797073}). Moreover, the balanced property allows us to study cohomological properties of the group, see \cref{virtually:cyclic:dimension}.

Note that the balanced property is not preserved under short exact sequences nor under graphs of groups. Indeed, for \(n \ge 2\), the Baumslag--Solitar group
\[
BS(1,n)=\langle a,t \mid t a t^{-1}=a^{n}\rangle
\]
is not balanced. Nevertheless, \(BS(1,n)\) decomposes as a semidirect product
\[
BS(1,n)\cong \mathbb Z\!\left[\tfrac{1}{n}\right]\rtimes \mathbb Z,
\]
where both factors \(\mathbb Z\!\left[\tfrac{1}{n}\right]\) and \(\mathbb Z\) are balanced. 
Moreover, \(BS(1,n)\) can also be realized as the fundamental group of a graph of groups with one vertex and one edge.

Given the importance of the balanced property, we study when the property of being balanced is preserved under short exact sequences, graphs of groups, and direct limits. For instance, we show that if the splitting of the fundamental group of a graph of groups is acylindrical in the sense of Sela and all vertex groups are balanced, then the fundamental group itself is balanced. As an application, we recover the fact that the fundamental groups of (possibly non-orientable) 3-manifolds are balanced. In what follows, we present these results in detail.


\begin{theorem}\label{graph:groups:condition:C:1}
Let $\mathcal{G} = (\Gamma, (G_{v})_{v \in V}, (G_{a})_{a \in A}, (\varphi_{a})_{a \in A})$ be a graph of groups with fundamental group $G$ and Bass–Serre tree $T$. Suppose that:
\begin{enumerate}[a)]
    \item for every $v \in V$, the vertex group $G_{v}$ is balanced; and
    \item For every elliptic element $h \in G$ of infinite order, the fixed-point set 
$T^{\langle h \rangle}$ does not contain the axis of any loxodromic element.
\end{enumerate}
Then $G$ is balanced.
\end{theorem}

Several special cases of \cref{graph:groups:condition:C:1} have been previously studied in the literature. 
For example, \cite{2015arXiv150905688B} treats the case of graphs of groups whose edge groups are cyclic.

The hypothesis imposed in \cref{graph:groups:condition:C:1} is fairly mild. 
In fact, we see in  \cref{applications:previus} many natural and widely occurring classes of graphs of groups satisfy this condition.

\begin{corollary}
Let $\mathcal{G} = (\Gamma, (G_{v})_{v \in V}, (G_{a})_{a \in A}, (\varphi_{a})_{a \in A})$ be a graph of groups with fundamental group $G$ and Bass–Serre tree $T$. Suppose that:
\begin{enumerate}[a)]
    \item for every $v \in V$, the vertex group $G_{v}$ is balanced; and
    \item the action of $G$ in $T$ is Sela-acylindrical. 
    \end{enumerate}
Then $G$ is balanced.
\end{corollary}

To prove \cref{graph:groups:condition:C:1}, we first show that, under suitable hypotheses, the balanced property is preserved under short exact sequences. In general, however, the balanced property is not preserved under short exact sequences; see \cref{ses:balanced:property}.




As an application of \cref{graph:groups:condition:C:1}, 
we recover the result that the fundamental groups of $3$--manifolds are balanced.

\begin{theorem}\label{balanced:3-manifolds}
The fundamental group of a compact, connected (possibly non-orientable) $3$-manifold is balanced.
\end{theorem}
\cref{balanced:3-manifolds} was proved by Jaco and Shalen for Haken \(3\)-manifolds \cite[Theorem VI.2.1]{Jaco:Shalen}, and the general case was later established in \cite{Shalen2001}. We prove \cref{balanced:3-manifolds} for closed \(3\)-manifolds.  
For the case of nonempty boundary, the result follows from \cite[Lemma~1.6]{matthias:stefan:henry} together with the fact that the balanced property passes to subgroups.

\subsection*{Outline of the paper}

In \cref{prelimi} we collect several preliminary facts concerning balanced groups that will be used throughout the paper. 

In \cref{preservation:balanced:property} we study the stability of the balanced property under short exact sequences, direct limits, and graphs of groups. In particular, we prove \cref{graph:groups:condition:C:1}. 

In \cref{applications:previus} we apply the results developed in the previous section to show that the fundamental groups of graph manifolds and $3$-manifolds are balanced. 

Finally, in \cref{virtually:cyclic:dimension} we establish upper bounds for certain geometric dimensions of poly-surface and poly-free groups. In particular, we prove \cref{poli:surface:th1} and \cref{poli:free:virtually:cy}.\\

\noindent{\bf Acknowledgments.} The first author was partially funded by the DGAPA-UNAM research grants IN114323 and IN106925. The second author was supported by the UNAM Postdoctoral Program (POSDOC) and by DGAPA-UNAM through grant PAPIIT~IN102426. We thank Rita Jiménez Rolland for helpful comments on a draft of this paper.

\section{Preliminaries}\label{prelimi}
In this section we collect several basic results concerning balanced groups that will be used in the subsequent sections.

\begin{definition}\cite[Condition 4.1]{MR2545612}
We say that a group \( G \) is \emph{balanced} or that it  satisfies \emph{condition C} if for every \( g,h \in G \) with \( |h| = \infty \) and \( k,l \in \mathbb{Z}\setminus\{0\} \), the equation \( gh^k g^{-1} = h^l \) implies that \( |k| = |l| \).
\end{definition}

\begin{example}\label{example:condition C}
    \begin{enumerate}
        \item Hyperbolic groups are a well-known example of balanced groups, see \cite{BH99}.
        \item More generally, since semi-hyperbolic groups satisfy an Algebraic Flat Torus Theorem (see \cite{BH99}), then semi-hyperbolic groups are balanced. In particular, CAT(0) groups are balanced.
        \item Using the same argument as the previous example, if $G$ is a group such that any infinite cyclic subgroup is undistorted (see \cite{DK18}), then $G$ is balanced.
        \item If a group is hyperbolic relative to subgroups that are balanced, then it is itself balanced (see Corollary 4.26 in \cite{MR2182268} or Proposition \ref{prop:Cinrelhyper} below).
    \end{enumerate}
\end{example}

It is well-known that being balanced is preserved under finite extensions and passes to subgroups. 
We state this as a lemma for future reference.

\begin{lemma}\cite[Lemma 4.14]{MR1760573}\label{conditionC:finite:extension}
If a group contains a finite-index subgroup that is balanced, then the whole group is balanced. 
Moreover, every subgroup of a balanced group is also balanced.
\end{lemma}

One of the main motivations for studying balanced groups is that it allows us to control the structure of the commensurator of virtually cyclic subgroups via normalizers of finite-index subgroups. In particular, under additional finiteness assumptions or the uniqueness of roots property, the commensurator can be realized precisely as a normalizer. The precise statements are given in the propositions below.

\begin{proposition}\label{conditionC:conmen:fg}
Let $G$ be a balanced group. Let $H$ be an infinite cyclic subgroup of $G$ and suppose that $N_G[H]$ is finitely generated. Then there is a finite index subgroup $H'\leq H$ such that $N_G[H]=N_G[H']=N_G(H')$.
\end{proposition}
\begin{proof}
Let $H=\langle h \rangle$ be an infinite cyclic subgroup of $G$ such that $N_G[H]$ is finitely generated by $g_1, \cdots, g_n$. Since $g_i\in N_G[H]$ we have that $g_ih^{k_i}g_i^{-1}=h^{l_i}$ for some integers $k_i,l_i$, by assumption $G$ is a balanced group, it follows that \( l_i = \pm k_i \). Let $k=k_1k_2\cdots k_n$. Note that $g_ih^kg_i^{-1}=h^{\pm k}$ for all $i=1,\cdots, n$. Then $N_G[H]\subseteq N_G( \langle h^{k} \rangle ) \subseteq N_G[ \langle h^{k} \rangle ]=N_G[H]$
the last equality holds because $\langle h^{k} \rangle$ is commensurable with $H$.
\end{proof}

\begin{proposition}\cite[Lemma 4.2]{MR2545612}\label{lemma:commensurator_cyclic}
Let \( G \) be a balanced group. Then, for any infinite cyclic subgroup \( C \) of $G$, we have a nested sequence of normalizers:
\[
N_G(C) \subseteq N_G(2!C) \subseteq N_G(3!C) \subseteq \cdots
\]
where 
\[
k!C = \{ h^{k!} \mid h \in C \}.
\]
Moreover
\[
N_G[C] = \bigcup_{k \geq 1} N_G(k!C).
\]

\end{proposition}

It was noted in \cite[proposition 3.4]{MR3797073} that if a group is balanced and satisfies the uniqueness of roots, then the commensurator of any infinite cyclic subgroup can be realized as the normalizer.

\begin{definition}[Uniqueness of Roots]
A group \( G \) has the property of uniqueness of roots if \( f, g \in G \) are such that \( f^n = g^n \) for some \( n \), then \( f = g \).
\end{definition}

\begin{theorem}\cite[proposition 3.4]{MR3797073}\label{conditionC:and:roots}
Let \( G \) be a balanced group that has a finite-index normal subgroup that satisfies the uniqueness of roots property. Then the commensurators of virtually cyclic subgroups of \( G \) are realized as normalizers.
\end{theorem}


\section{Preservation of the balanced property}\label{preservation:balanced:property}

In this section we investigate the stability of the balanced property under short exact sequences, direct limits, and graphs of groups. In particular, in this section we prove \cref{graph:groups:condition:C:1}.

\subsection*{Balanced groups and short exact sequences}\label{ses:balanced:property}

The balanced property is not preserved under short exact sequences in general.
Indeed, for $n \ge 2$ the Baumslag--Solitar group
\[
BS(1,n)=\langle a,t \mid t a t^{-1}=a^{n}\rangle
\]
fails to be balanced.
On the other hand, $BS(1,n)$ admits a decomposition as a semidirect product
\[
BS(1,n)\cong \mathbb Z\!\left[\tfrac{1}{n}\right]\rtimes \mathbb Z,
\]
where both factors $\mathbb Z\!\left[\tfrac{1}{n}\right]$ and $\mathbb Z$ are balanced groups.

Nevertheless, we show that under suitable additional assumptions on the kernel, the balanced property is preserved under short exact sequences. To this end, we introduce two stronger conditions ensuring that the balanced property is preserved under short exact sequences.

\begin{definition}
Let \( G \) be a group. We say that \( G \) satisfies \emph{Condition A} if for every automorphism \( \varphi \in \operatorname{Aut}(G) \) and every element \( g \in G \) of infinite order, the relation
\[
\varphi(g^k) = g^\ell
\]
for some \( k, \ell \in \mathbb{Z}\setminus\{0\} \) implies \( |k| = |\ell| \).
\end{definition}
 In \cite[Lemma~5.14(v)]{LW12}, it was proved that polycyclic groups satisfy Condition~A.

The following was proved by Lück-Weirmann in \cite[Proof of Lemma 5.14 item (\textit{v}) ]{LW12}
\begin{lemma}\label{propertyA:sec:LW}
Consider a short exact sequence of  groups
\[
1 \to N \to G \xrightarrow[]{\varphi} Q \to 1.
\]
If \( Q \) satisfies Condition A and \( N \) is characteristic in $G$ and satisfies Condition A, then \( G \) also satisfies Condition A.
\end{lemma}

\begin{corollary}\label{cor:virtually-surface}
Let \( G \) be a finitely generated virtually surface group. Then
\begin{enumerate}[a)]
    \item \(G\) satisfies Condition~A;
    \item the commensurators of virtually cyclic subgroups of \(G\) are realized as
    normalizers.
\end{enumerate}
\end{corollary}

\begin{proof}
We first prove item~(a). Let \(H\le G\) be a surface subgroup of finite index and consider
\[
N = \bigcap_{\varphi \in \Aut(G)} \varphi(H).
\]
Since \(G\) is finitely generated, it has only finitely many subgroups of a given finite
index. Consequently, the above intersection is taken over finitely many subgroups, and
hence \(N\) has finite index in both \(G\) and \(H\). In particular, \(N\) is itself a
surface group. Moreover, by construction \(N\) is characteristic in \(G\), as it is
invariant under all automorphisms of \(G\). The claim now follows from
Lemma~\ref{propertyA:sec:LW}, applied to the characteristic finite-index subgroup \(N\).

For item~(b), note that the finite-index surface subgroup may be chosen to be normal in
\(G\). The conclusion then follows from item~(a) together with
\cref{conditionC:and:roots}.
\end{proof}

\medskip

\begin{remark}
The same argument applies to the fundamental group of a
closed hyperbolic \(3\)-manifold.

\end{remark}

Our next result can be viewed as an analogue of \cref{propertyA:sec:LW} for balanced groups. In contrast with the situation for Condition~A, we show that the
characteristic assumption on the kernel can be dropped, provided that the kernel
satisfies the stronger Condition~A.
\begin{theorem}\label{sec:conditionC}
Consider a short exact sequence of groups
\[
1 \to N \to G \xrightarrow[]{p} Q \to 1.
\]
If \( Q \) is balanced and \( N \) satisfies Condition~A, then \( G \) is balanced.
\end{theorem}
\begin{proof}
Let \( g,h \in G \) with \( |h| = \infty \) and \( k,l \in \mathbb{Z}\setminus\{0\} \) such that \( gh^k g^{-1} = h^l \).
 We have two cases $p(h)$ has infinite order or not. In the first case we have $p(g)p(h)^k p(g)^{-1} = p(h)^l$  the claim follows from the fact that $Q$ is balanced.

 In the second case we have that $p(h)^n=1$ for some integer $n$.  As \( N \) is normal in \( G \), conjugation by \( g \) leaves $N$ invariant. This implies that conjugation by $g$ induces an automorphism \(\psi: N \to N\). Then $\psi(h^n)^k=(h^n)^l$, since $N$ satisfies Condition A it follows that $|k|=|l|$.
 \end{proof}


\begin{definition}[Condition MAX]
We say that a group $G$ satisfies \emph{Condition MAX} if for every infinite cyclic
subgroup $C \le G$ there exists a \emph{unique} maximal virtually cyclic subgroup
$M \le G$ such that $C \le M$.
\end{definition}

\begin{remark}\label{remark:examples MAX}
    \begin{enumerate}
        \item Note that any hyperbolic group has property (MAX) (see \cite[Chapitre 8, Théorème 30]{MR1086648}).
        \item Free groups of arbitrary rank.
        \item Free abelian groups of arbitrary rank.
        
    \end{enumerate}
\end{remark}

\begin{lemma}\label{virtually:Z:conditionA}
Let $G$ be a virtually cyclic group. Then $G$ satisfies Condition~A.
\end{lemma}

\begin{proof}
If $G$ is finite, the statement is immediate. Suppose therefore that $G$ is infinite and virtually cyclic. Then $G$ contains a finite-index infinite cyclic subgroup. Since $G$ is finitely generated, then  $G$ contains a characteristic subgroup $H \cong \mathbb{Z}$ of finite index.

Thus we obtain a short exact sequence
\[
1 \longrightarrow H \longrightarrow G \longrightarrow G/H \longrightarrow 1,
\]
where $H \cong \mathbb{Z}$ and $G/H$ is finite. The result now follows from \cref{propertyA:sec:LW}.
\end{proof}

\begin{theorem}\label{proposition:MAX implies A}
Let $G$ be a group satisfying Condition~MAX. Then $G$ satisfies Condition~A.
\end{theorem}

\begin{proof}
Let $h \in G$ be an element of infinite order, and let $\varphi \in \Aut(G)$ such
that
\[
\varphi(h^m) = h^n
\]
for some $m,n \in \mathbb{Z} \setminus \{0\}$.
By hypothesis, the infinite cyclic subgroup $\langle h \rangle$ is contained in a
unique maximal virtually cyclic subgroup $M \le G$.

Since $\varphi$ is an automorphism, the image $\varphi(M)$ is again a maximal
virtually cyclic subgroup of $G$. Moreover,
\[
\langle h^n \rangle
= \varphi(\langle h^m \rangle)
\le \varphi(M).
\]
Given that $\langle h^n \rangle \leq \langle h \rangle \leq M$, by uniqueness of the maximal virtually cyclic subgroup containing
$\langle h^n \rangle$, it follows that
\[
\varphi(M) = M.
\]
In particular, $\varphi$ restricts to an automorphism
\[
\varphi|_M \in \Aut(M).
\]

Since $M$ is virtually cyclic,  it follows from \cref{virtually:Z:conditionA} that $M$
satisfies Condition~A. Applying this condition to the equality
$\varphi(h^m) = h^n$ inside $M$, we conclude that
\[
|m| = |n|.
\]
\end{proof}

 \begin{corollary}\label{propertyC:sec}
Consider a short exact sequence of  groups
\[
1 \to N \to G \xrightarrow[]{\varphi} Q \to 1.
\]
If \( Q \)  is balanced and \( N \) has 
property (MAX), then \( G \) is balanced.
\end{corollary}

In the following sections we present several applications of \cref{sec:conditionC} and \cref{propertyC:sec}. 
For instance, from \cref{sec:conditionC} we deduce that the fundamental group of a Seifert 3-manifold is balanced (see \cref{conditionC:Seifert}). 
Similarly, using \cref{propertyC:sec} we show that normally poly-hyperbolic groups are balanced (see \cref{hyperbolic:groups:propertyC}).

\subsection*{Balanced groups and direct limits}
In this subsection we study the behavior of the balanced property under direct limits of groups.  We consider direct limits indexed by directed sets in the
following sense.

\begin{definition}
A partially ordered set $(I,\le)$ is called \emph{directed} if for any
$\alpha,\beta\in I$ there exists $\gamma\in I$ such that
\[
\alpha\le \gamma \quad \text{and} \quad \beta\le \gamma.
\]
\end{definition}

\begin{proposition}\label{balanced:direc:limits}
Let $(I,\le)$ be a directed set (i.e., for any $\alpha,\beta\in I$ there exists $\gamma\in I$ with $\gamma\ge \alpha,\beta$), and let
\[
G = \varinjlim_{i\in I} G_i
\]
be the direct limit of a directed system of groups $\{G_i\}_{i\in I}$, each of which is balanced. Then $G$ is balanced.
\end{proposition} 

\begin{proof}
Let $h\in G$ be of infinite order and suppose there exists $g\in G$ and integers $k,\ell\in\mathbb Z\setminus\{0\}$ such that
\[
gh^{k}g^{-1}=h^{\ell}\quad\text{in }G.
\]
We must show that $|k|=|\ell|$. Pick indices $\alpha,\beta\in I$ and representatives $h_\alpha\in G_\alpha$, $g_\beta\in G_\beta$ mapping to $h,g$ in $G$. 
By the directedness of $(I,\le)$ there is $\gamma\ge \alpha,\beta$ such that the equality
\[
g_\gamma h_\gamma^{k}g_\gamma^{-1}=h_\gamma^{\ell}
\]
already holds in $G_\gamma$, where $h_\gamma,g_\gamma$ denote the images of $h_\alpha,g_\beta$ in $G_\gamma$, note that $h_\gamma$ has infinite order, otherwise $h$ would have finite order in $G$, a contradiction. Since $G_\gamma$ is balanced, it follows that $|k|=|\ell|$. This proves that $G$ is balanced.
\end{proof}

\begin{remark}
Every group is the direct limit of its finitely generated subgroups. Hence, if a group $G$ is locally balanced, that is, every finitely generated subgroup of $G$ is balanced, then $G$ itself is balanced. In particular, this applies to classes built locally from groups known to be balanced, such as free groups, locally hyperbolic groups, and locally finite groups (the latter vacuously, since they contain no elements of infinite order).
\end{remark}

\begin{proposition}\label{prop:nilpotent is balanced}
Every virtually nilpotent group is balanced.
\end{proposition}
\begin{proof}
By \cref{conditionC:finite:extension}, it suffices to prove the claim for nilpotent groups.  
Since every group is the direct limit of its finitely generated subgroups, it follows from \cref{balanced:direc:limits} that it is enough to show that every finitely generated nilpotent group is balanced.  

It is well-known that every finitely generated nilpotent group is polycyclic (see for example \cite[Theorem 1.8]{pengitore2025}), and polycyclic groups are balanced. The result follows.
\end{proof}

\subsection*{Balanced groups and graph of groups}

Being balanced is not necessarily preserved under graphs of groups. More precisely, even if all vertex groups and edge groups are balanced, the fundamental group of the graph of groups need not be balanced. A concrete example is given by the Baumslag--Solitar group
\[
BS(1,n)=\langle a,t \mid t a t^{-1}=a^{n}\rangle,
\]
which can be seen as the fundamental group of a loop of groups where all vertex and edge groups are $\mathbb{Z}$.

Thus, we require more hypotheses to produce balanced groups via graph of groups. For the sake of completeness, we recall some basic notions on isometries of a tree, needed for the formulation of the proposition.

\subsection*{Bass-Serre Theory}

We recall some basic concepts from Bass-Serre theory; for further details, see \cite{Se03}. In general, Bass-Serre theory studies groups via their actions on trees and one of its more iconic tools is what is called a graph of groups.  To see this, first we define a graph as a tuple $\Gamma = (V,E, \partial_{0}, \overline{\cdot})$ with $V$ and $E$ being sets (whose elements are called vertices and edges respectively), $\partial_{0}: E \to V$ is the \emph{initial vertex} map, and $\overline{\cdot}: E \to E$ is an involution without fixed points called the \emph{inverse edge} map. With this we can also define $\partial_{1} := \partial_{0}(\overline{\cdot})$ as the \emph{terminal vertex} map.

Given a graph $\Gamma$, a \emph{path} is a finite (possibly empty) sequence $(e_{1}, \ldots, e_{n})$ where for all $i$ we have that $e_{i}$ is an edge and $\partial_{1}(e_{i}) = \partial_{0}(e_{i+1})$. Then, $\Gamma$ is \emph{connected} if for any vertices $u,v \in V$ there exists a path $c = (e_{1}, \ldots, e_{n})$ such that $u = \partial_{0}(e_{1})$ and $v = \partial_{1}(e_{n})$.

A path $c = (e_{1}, \ldots, e_{n})$ is \emph{closed} if $\partial_{1}(e_{n}) = \partial_{0}(e_{1})$. Also, $c$ has \emph{backtracking} if there exists $i$ such that $\overline{e_{i}} = e_{i+1}$.

Now, a graph $\Gamma$ is simplicial if for all $u,v \in V$, the set $\{e \in E: \partial_{0}(e)=u, \partial_{1}(e) = v\}$ has cardinality at most one.

Then a \emph{tree} in this context is a non-empty, connected and simplicial graph $T$ such that any closed path is either constant or has backtracking.

Given a group $G$ and a tree $T$, we say that \emph{$G$ acts (without inversion) on $T$} if there exist an action of $G$ on $V$ and an action of $G$ on $E$ such that for all $g \in G$ and all $e \in E$ we have:
\begin{enumerate}
    \item $\partial_{0}(g \cdot e) = g \cdot \partial_{0}(e)$,
    \item $\overline{g \cdot e} = g \cdot \overline{e}$, and
    \item $g \cdot e \neq \overline{e}$.
\end{enumerate}

Given one such action of $G$ on $T$, we can summarize this action as follows: A \emph{graph of groups} $\mathcal{G}$ is a tuple $(\Gamma, (G_{v})_{v \in V}, (G_{e})_{e \in E}, (\varphi_{e})_{e \in E})$ such that:
\begin{itemize}
    \item $\Gamma$ is a graph,
    \item for all $v \in V$, $G_{v}$ is a group,
    \item for all $e \in E$, $G_{e}$ is a group and $G_{e} = G_{\overline{e}}$, and
    \item for all $e \in E$, $\varphi_{e}:G_{e} \to G_{v}$ is a monomorphism.
\end{itemize}

If $G$ is a group acting on a tree $T$, then there is a natural graph of groups with underlying graph the quotient graph $G\backslash\Gamma$, fixing a choice of representatives of conjugacy classes of vertex stabilizers, doing similarly for edge stabilizers, and letting the morphisms be the natural inclusion maps possibly pre-composed by the ad hoc conjugation.

Moreover, one can recover the tree, the group and the action from the graph of groups. Given a graph of groups $\mathcal{G}$, there are two objects associated to $\mathcal{G}$, its fundamental group $\pi_{1}(\mathcal{G})$ and its Bass-Serre tree $T_{\mathcal{G}}$. With this, we have the following structure theorem.

\begin{theorem}[Theorem 13 in \cite{Se03}]
    Let $G$ be a group and $T$ be a tree such that $G$ acts without inversion on $T$. Then, there is an isomorphism $\psi: \pi_{1}(\mathcal{G}) \to G$, and $T_{\mathcal{G}}$ is $\psi$-equivariantly isomorphic to $T$, where $\mathcal{G}$ is the graph of groups associated to the action of $G$ on $T$.
\end{theorem}

We refer the reader to \cite[Section 5]{Se03} for more on this topic, or in \cite[Theorem 4.3]{MR0564422} for a more topological viewpoint.

The last missing ingredients for the proofs in this section are understanding the behaviour of the elements of $G$.

Fixing a group $G$, a tree $T$ and an action (without inversions) of $G$ on $T$, we say that an element $g \in G$ is \emph{elliptic} if there exists a $v \in V$ such that $g \cdot v = v$. Equivalently, if $g$ is an element of a conjugate of $G_{v}$ for some $v \in V$. In this case, the set of fixed points of $g$ induces a subtree denoted by $T^{\langle g \rangle}$.

On the other hand, we say $g \in G$ is \emph{loxodromic} (also called sometimes hyperbolic) if there exists a unique geodesic line $\ell$ (the isomorphic image of a bi-infinite path embedded on $T$) called the \emph{axis of $g$}, that is invariant such that $g$ acts on $\ell$ by translations. Equivalently, if $g$ is not elliptic.

\begin{theorem}\label{graph:groups:condition:C}
Let $\mathcal{G} = (\Gamma, (G_{v})_{v \in V}, (G_{a})_{a \in A}, (\varphi_{a})_{a \in A})$ be a graph of groups with fundamental group $G$ and Bass–Serre tree $T$. Suppose that:
\begin{enumerate}[a)]
    \item for every $v \in V$, the vertex group $G_{v}$ is balanced; and
    \item For every elliptic element $h \in G$ of infinite order, the fixed-point set 
$T^{\langle h \rangle}$ does not contain the axis of any loxodromic element.
\end{enumerate}
Then $G$ is balanced.
\end{theorem}
\begin{proof}
     Given $g,h \in G$ with $h$ of infinite order such that there exist $k,l \in \mathbb{Z}\setminus\{0\}$ satisfying that $g h^{k} g^{-1} = h^{l}$, we need to prove that $|k| = |l|$.
    
    Now we divide the proof into two cases.

\emph{Case 1. $h$ is an elliptic automorphism of $T$:} 
Since $g h^{k} g^{-1} = h^{l}$, it follows that $g$ preserves $T^{\langle h \rangle}$. 
Moreover, by assumption $T^{\langle h \rangle}$ does not contain the axis of any loxodromic element. Therefore $g$ is elliptic, and hence it fixes some vertex $v \in T^{\langle h \rangle}$.
Since, by assumption, $G_{v}$ is balanced and $g,h \in G_{v}$, it follows that $|k| = |l|$.

    \emph{Case 2. $h$ is a loxodromic automorphism of $T$:} 
In this case, there exists a unique bi-infinite geodesic $\ell \subset T$ invariant under $h$, 
on which $h$ acts by translations. 
It follows that $g h^{k} g^{-1}$ leaves $g \cdot \ell$ invariant, acting again by translations. 
Since $g h^{k} g^{-1} = h^{l}$ and $h^{l}$ acts on the unique invariant geodesic $\ell$, 
we deduce that $g \cdot \ell = \ell$. 
Thus $g,h \in \operatorname{Stab}_{set}(\ell)$.

    Since $\ell$ is a bi-infinite geodesic then $Aut(\ell) \cong D_{\infty}$, and then we obtain a representation $\rho\colon stab_{set}(\ell) \to D_{\infty}$, by assumption, the fixed-point set of every elliptic element $h \in G$ of infinite order does not contain the axis of any loxodromic element, it follows that the kernel is a torsion group. Since the property of being balanced is inherited by subgroups, and that torsion groups satisfy property $(MAX)$ vacuously, we have that the short exact sequence $$1 \to ker(\rho) \to stab_{set}(\ell) \to \rho(stab_{set}(\ell)) \to 1$$ satisfies the conditions of Corollary \ref{propertyC:sec}.
    Hence, $stab_{set}(\ell)$ is balanced and therefore $|k| = |l|$.
\end{proof}


\begin{remark}\label{remark:BSTconditions}
    While it seems that item $(b)$ is fairly restrictive, note that there are many cases which satisfy it:
\begin{enumerate}
    \item If all edge groups are finite.
    \item If the action of $G$ on $T$ is WPD, then this is satisfied (see \cite{MR1914565} for the definition of WPD).
    \item In particular, if the action is Sela-acylindrical (i.e. there exists a $k \geq 0$ such that for every elliptic element $h \in G$ we have that $T^{\langle h\rangle}$ has diameter at most $k$).
\end{enumerate}
\end{remark}

An immediate consequence of Remark \ref{remark:BSTconditions} (3), we have the following result.

\begin{corollary}\label{cor:CinAcylTrees}
Let $\mathcal{G} = (\Gamma, (G_{v})_{v \in V}, (G_{a})_{a \in A}, (\varphi_{a})_{a \in A})$ be a graph of groups with fundamental group $G$ and Bass–Serre tree $T$. Suppose that:
\begin{enumerate}[a)]
    \item for every $v \in V$, the vertex group $G_{v}$ is balanced; and
    \item the action of $G$ in $T$ is Sela-acylindrical. 
    \end{enumerate}
Then $G$ is balanced.
\end{corollary}

The above corollary provides a mechanism for constructing a wide class of explicit examples of balanced groups. In particular, it applies to groups arising from acylindrical actions on trees. As an application, in the following section we show that the fundamental groups of both irreducible extended graph manifolds and 3-manifolds are balanced.

Also, using Remark \ref{remark:BSTconditions} (1), we have the following consequence.
    
\begin{corollary}
Let \( G \) be a virtually free group. Then 
\begin{enumerate}[a)]
    \item $G$ is balanced.
    \item the
commensurators of virtually cyclic subgroups of \( G \) are realized as normalizers.
\end{enumerate}
\end{corollary}

\begin{proof}
By \cite{MR0564422}, the group $G$ can be realized as the fundamental group of a graph of groups whose vertex and edge groups are finite. It then follows from \cref{graph:groups:condition:C} that $G$ is balanced.

On the other hand, \( G \) contains a free subgroup of finite index, and since every free group has uniqueness of roots, the result follows from \cref{conditionC:and:roots}.
\end{proof}

\section{Some applications of the previous results}\label{applications:previus}

In this section we present several applications of the preceding results. In particular, we prove that the fundamental groups of $3$-manifolds and graph manifolds are balanced.

\subsection{Riemannian manifolds and graph manifolds}

As was mentioned in Example \ref{example:condition C}, hyperbolic groups are known to be balanced, and under the right conditions relatively hyperbolic groups are also balanced.

\begin{proposition}[Balanced property in relatively hyperbolic groups]\label{prop:Cinrelhyper}
Let $G$ be a group that is hyperbolic relative to a finite collection of subgroups
$\mathcal H=\{H_1,\dots,H_n\}$.
Assume that each peripheral subgroup $H_i$ is balanced.
Then $G$ is balanced.
\end{proposition}

\begin{proof}
Let $h\in G$ be an element of infinite order and let $g\in G$ such that
\[
g h^m g^{-1}=h^n
\]
for some $m,n\in\mathbb Z\setminus\{0\}$.
We must prove that $|m|=|n|$. 

First suppose that there exists $x\in G$ and $i\in\{1,\dots,n\}$ such that
\[
h\in x H_i x^{-1}.
\]
The equality $g h^m g^{-1}=h^n$ implies that
\[
\langle h^n\rangle \subseteq x H_i x^{-1}\cap g x H_i x^{-1} g^{-1}.
\]
Since $h$ has infinite order, this intersection is infinite. Moreover,
\[
x^{-1}( x H_i x^{-1}\cap g x H_i x^{-1} g^{-1})x
= H_i\cap x^{-1}g x H_i x^{-1} g^{-1}x,
\]

which implies that $H_i\cap x^{-1}g x H_i x^{-1} g^{-1}x$ is also infinite.

Because peripheral subgroups are almost malnormal in relatively hyperbolic groups,
it follows that $x^{-1}gx \in H_i$. Hence,
\[
g\in x H_i x^{-1}.
\]
Thus $g,h\in x H_i x^{-1}$ and the equality $g h^m g^{-1}=h^n$ holds inside the subgroup
$x H_i x^{-1}$. Since being balanced is invariant under conjugation and $H_i$ is balanced,
we conclude that $|m|=|n|$.

Finally, in the case where $h$ does not belong to any conjugate of a peripheral subgroup, the statement was proved in \cite[Corollary 4.26]{MR2182268} (recall that in \cite{MR2182268} a hyperbolic element is an element that cannot be conjugated into some $H_i$). For the sake of completeness, we give a quick summary of the proof is that in \cite{MR2182268}: A \emph{relative translation length} ($\tau^{rel}$) is introduced, which satisfies the same properties as the algebraic translation length in \cite{BH99} (it is invariant under conjugation and transforms powers into the absolute value of the sums). Moreover, in \cite{MR2182268} they prove that the relative translation length of any element that cannot be conjugated into some $H_i$ is positive (Theorem 4.23). Thus we have
$$\begin{array}{rcl}
  |n|\tau^{rel}(h)   & = &  \tau^{rel}(h^{n}) \\
     & = &  \tau^{rel}(gh^{m}g {-1}) \\
     & = &  \tau^{rel}(h^m) \\
     & = & |m|\tau^{rel}(h)
\end{array}$$
Implying that $|m| = |n|$ as desired.

\end{proof}

With this result, we have that the fundamental group of any complete finite-volume Riemannian manifold with pinched negative curvature, is balanced.

\begin{corollary}\label{cor:CinNonPosRiemannian}
    Let $M$ be a complete finite-volume Riemmanian manifold with pinched negative curvature. Then, $\pi_{1}(M)$ is balanced.
\end{corollary}
\begin{proof}
    Due to Theorem 4.11 in \cite{MR1650094}, we know that $\pi_{1}(M)$ is hyperbolic relative to their cusp subgroups. Moreover, we know that these cusp subgroups are finitely generated virtually nilpotent groups. Then, by Propositions \ref{prop:nilpotent is balanced} and \ref{prop:Cinrelhyper}, we have the result.
\end{proof}

In a more general setting, we can use the previous results to see that irreducible (extended) graph manifolds (see \cite{MR3444648}) have also a balanced fundamental group. For the sake of completeness, we include the relevant definitions.

\begin{definition}[cf. \cite{MR3444648}]
    Let $M$ be a compact smooth $n$-manifold wth $n \geq 3$.
    \begin{enumerate}
        \item We say $M$ is a \emph{graph manifold} if it can be constructed in the following way.
        \begin{enumerate}
            \item For every $i = 1, \ldots, r$, let $N_{i}$ be complete finite-volume non-compact hyperbolic $n_{i}$-manifold with $3 \leq n_{i} \leq n$ and toric cusps, and denote by $\overline{N_{i}}$ the resulting $n_{i}$-manifold from obtained by truncating its cusps
            \item For each $i = 1, \ldots, r$, define $V_{i} := \overline{N_{i}} \times T^{n-n_{i}}$, where $T^k$ is the $k$-dimensional torus.
            \item Fix a pairing of some boundary components of the $V_i$'s and glue the paired boundary components using affine diffeomorphisms of the boundary tori, so as to obtain a connected manifold of dimension $n$.
        \end{enumerate}
        Note that $\partial M$ is either empty or consists of tori. The submanifolds $V_i$ are called \emph{pieces}, and given the obvious trivial bundle structure, the $\overline{N_{i}}$ are called the \emph{base} of $V_i$, while sets of the form $\{\ast\} \times T^{n-n_i}$ are called \emph{fibers}. Also, the boundary tori of the $V_i$'s that are glued together are called \emph{internal walls} of $M$. If $V^+$ and $V^-$ are two pieces glued by an internal wall, we say they are \emph{adjacent}.
        \item An \emph{extended graph manifold} is an $n$-manifold constructed as above where it is permitted to have pieces with base a complete non-compact hyperbolic surface (thus, the corresponding fibers are $(n-2)$-dimensional) called \emph{surface pieces}, such that there are no adjacent surface pieces.
        \item Let $V^+$ and $V^-$ be adjacent pieces with corresponding paired boundary tori $T^+$ and $T^-$, and $T$ be the corresponding internal wall. Let also $\psi: T^+ \to T^-$ be the gluing diffeomorphism, $\psi_*$ be the corresponding $\pi_1$-automorphism, and $G_+$ and $G_-$ be the respective subgroups of $\pi_{1}(T^+)$ and $\pi_{1}(T^-)$ corresponding to the fiber subgroups. We say that the two pieces have \emph{transverse fibers} if $\psi_*(G_+) \cap G_-$ is trivial.
        \item An extended graph manifold is \emph{irreducible} if every pair of adjacent pieces has transverse fibers along every interior wall.
    \end{enumerate}
\end{definition}

Note that all extended graph manifolds have a natural decomposition as graph of spaces, and as such their fundamental group is the fundamental group of the induced graph of groups. One of the results from \cite{MR3444648} that is most relevant for this work is the following.

\begin{proposition}[6.4 in \cite{MR3444648}]\label{proposition:irred if and only if acyl}
    Let $M$ be an extended graph manifold with at least one interior wall. Then, $M$ is irreducible if and only if the action of $\pi_{1}(M)$ on its induced Bass-Serre tree is Sela-acylindrical.
\end{proposition}

A consequence of this result and Corollaries \ref{propertyC:sec}, \ref{cor:CinAcylTrees} and \ref{cor:CinNonPosRiemannian} is the following.

\begin{theorem}
    Let $M$ be an irreducible extended graph manifold. Then $\pi_{1}(M)$ is balanced.
\end{theorem}

\begin{proof}
    Using Proposition \ref{proposition:irred if and only if acyl} and Corollary \ref{cor:CinAcylTrees}, we need only prove that $\pi_{1}(V_{i})$ are balanced for all $1 \leq i \leq r$.

    For each $i$, we have that $\overline{N_{i}}$ is a complete compact hyperbolic manifold. Hence, by Corollary \ref{cor:CinNonPosRiemannian}, $\pi_{1}(\overline{N_{i}})$ s balanced. Then, since $\pi_{1}(T^{n-n_{i}}) \cong \Z^{n-n_{i}}$ we have that $\pi_{1}(T^{n-n_{i}})$ has property (MAX).

    Therefore, by Corollary \ref{propertyC:sec}, we have that $\pi_{1}(V_{i}) \cong \pi_{1}(\overline{N_{i}}) \oplus \pi_{1}(T^{n-n_{i}})$ is balanced, which finishes the proof.
\end{proof}

\subsection{3-manifolds and balanced groups}
In this subsection we prove \cref{balanced:3-manifolds}.
We begin by recalling several fundamental results about $3$-manifolds that we use in this section. A $3$-manifold $M$ is said to be \emph{prime} if whenever it admits a connected sum decomposition $M = M_{1} \# M_{2}$, then one of the factors is the $3$--sphere $S^{3}$. 

One of the cornerstones of $3$-manifold topology is the fact that every closed, orientable $3$-manifold decomposes as a connected sum of prime $3$-manifolds. The existence of this decomposition is due to Kneser~\cite{Kneser1929}, while uniqueness was established by Milnor~\cite{Milnor:prime:decomposition}.

\begin{theorem}[Prime Decomposition Theorem]\label{split:prime}
Let $M$ be a closed, connected, orientable $3$--manifold. Then $M$ admits a decomposition
\[
M \;\cong\; P_{1} \# \cdots \# P_{n},
\]
where each $P_{i}$ is prime. Moreover, this decomposition is unique up to homeomorphism and permutation of the factors.
\end{theorem}

The above theorem reduces the study of closed $3$-manifolds to the prime case. A further celebrated structure theorem is the Jaco--Shalen--Johannson decomposition (JSJ decomposition), which asserts that any prime $3$--manifold can be cut along a canonical family of incompressible tori into pieces that are either hyperbolic or Seifert fibered.  A \emph{Seifert fibered $3$--manifold} is an orientable compact $3$--manifold $M$ together with a decomposition of $M$ into disjoint simple closed curves, called \emph{fibers}, such that each fiber has a neighborhood homeomorphic to a fibered solid torus.

Following the work of Jaco--Shalen~\cite{Jaco:Shalen} and Johannson~\cite{Johannson}, and in light of the Geometrization Theorem proved by Perelman~\cite{Perelman1,Perelman2} building on Thurston’s  program~\cite{Thurston1}, the JSJ decomposition can be stated as follows.

\begin{theorem}[JSJ Decomposition Theorem]\label{thm:descomposition:JSJ}
Let $M$ be a closed, orientable, prime $3$--manifold. Then there exists a canonical (up to isotopy) collection $\mathcal{T}\subset M$ of incompressible tori, each properly embedded, $\pi_1$-injective, and with trivial normal bundle, such that each component of $M\setminus \mathcal{T}$ is either hyperbolic or Seifert fibered. 
\end{theorem}

For surfaces, the Uniformization Theorem ensures that every closed surface admits one of the constant curvature geometries $\mathbb{S}^2$, $\mathbb{E}^2$, or $\mathbb{H}^2$ (see, e.g.,~\cite{John:Stillwell,John:W:Morgan}). By contrast, Thurston’s Geometrization Conjecture, now a theorem due to Perelman established that there are precisely eight model geometries for $3$--manifolds: hyperbolic $\mathbb{H}^3$, Euclidean $\mathbb{E}^3$, spherical $\mathbb{S}^3$, Sol, Nil, $\mathbb{S}^2\times\mathbb{E}$, $\mathbb{H}^2\times\mathbb{E}$, and the universal cover $\widetilde{SL_2(\mathbb{R})}$ (see, e.g.,~\cite[Theorem 5.1]{Sc83}). 



\begin{theorem}
The fundamental group of a compact $3$-manifold is balanced.
\end{theorem}

\begin{proof}
Since every non-orientable $3$-manifold admits an orientable double cover, it follows from \cref{conditionC:finite:extension} that it is enough to prove the claim for orientable $3$-manifolds.
Let $M$ be an orientable $3$-manifold.  By the prime decomposition theorem (\cref{split:prime}) we may write
\[
M = P_1 \# P_2 \# \cdots \# P_n,
\]
and consequently
\[
\pi_1(M) \;\cong\; \pi_1(P_1) * \pi_1(P_2) * \cdots * \pi_1(P_n).
\]
By \cref{graph:groups:condition:C}, it suffices to prove that each $\pi_1(P_i)$ is balanced. Now let $P$ be an orientable prime $3$--manifold. Then either $P \cong S^1 \times S^2$, or $P$ is an irreducible $3$--manifold (see for example \cite[Proposition 1.4]{hatcher3mfd}). In the former case, we have $\pi_1(P) \cong \mathbb{Z}$, which is balanced. In the latter case, by the JSJ decomposition theorem (\cref{thm:descomposition:JSJ}), there exists a collection of embedded tori in $P$ such that each component of the decomposition is either a hyperbolic or a Seifert fibered manifold. In particular, this induces a splitting of $\pi_1(P)$ as a graph of groups, where the vertex groups are isomorphic to the fundamental groups of the JSJ-pieces $\pi_1(L_i)$, and the edge groups are isomorphic to $\mathbb{Z}^2$, corresponding to the fundamental groups of the JSJ-tori. By \cite[Lemma 2.4]{MR2660565}, this splitting is acylindrical. Hence by \cref{graph:groups:condition:C} it suffices to show that each $\pi_1(L_j)$ is balanced. 
This will be established in \cref{conditionC:Seifert} and \cref{conditionC:hyperbolic:manifold}.
\end{proof}

To prove Corollary \ref{conditionC:hyperbolic:manifold} and Proposition \ref{conditionC:Seifert}, we use the following theorem.

\begin{theorem}\label{conditionC:hyperbolic:n-manifolds}
    Fixing $n \geq 2$, let $G$ be a discrete subgroup of $Isom(\mathbb{H}^{n})$. Then, $G$ is balanced.
\end{theorem}
\begin{proof}
    Given Lemma \ref{conditionC:finite:extension}, we need only consider the case for $G \leq Isom^{+}(\mathbb{H}^{n})$. Let $g,h \in G$ be such that for some $k, l \in \mathbb{Z} \setminus\{0\}$ we have that $h^{k} = gh^{l}g^{-1}$. Again, by Lemma \ref{conditionC:finite:extension}, it suffices to prove that $g$ and $h$ are contained in a virtually abelian subgroup of $G$.

    Let $P \subset \partial\mathbb{H}^{n}$ be the set of fixed points of $h$. Then, $P$ is a non-empty set with at most two points, and it is also the set of fixed points of any non-trivial power of $h$. In particular, $P$ is the set of fixed points of both $h^{k} = gh^{l}g^{-1}$ and $h^{l}$. But, the set of fixed points of $gh^{l}g^{-1}$ is $g \cdot P$. This implies that if $H$ is the setwise stabilizer in $G$ of $P$, then $g,h \in H$. Since $P$ is finite, then $H$ is a discrete elementary subgroup of $G$ (Definition 3.22 of \cite{MR1792613}). By Section 4.6 of \cite{MR1792613}, this implies that $H$ is virtually abelian.
\end{proof}


\begin{corollary}\label{conditionC:hyperbolic:manifold}
 The fundamental group of a hyperbolic 3-manifold is balanced.   
\end{corollary}

  Let $M$ be a Seifert fibered $3$--manifold. Given a decomposition of $M$ into circles, one can collapse each fiber to a point. The resulting space is a $2$--orbifold $B$, called the \emph{base orbifold} of $M$; see \cite[Section~3]{Sc83}. The orbifold $B$ has an \emph{orbifold fundamental group} $\pi_1^{\mathrm{orb}}(B)$, which in general does not coincide with the fundamental group of its underlying topological space. A $2$--orbifold $B$ falls into exactly one of the following classes, depending on the geometry of its universal orbifold cover: bad, spherical, hyperbolic, or Euclidean.
\begin{proposition}\label{conditionC:Seifert}
Let $G$ be the fundamental group of a Seifert $3$-manifold. Then 
\begin{enumerate}[a)]
    \item $G$ is balanced,
    \item  the
commensurators of virtually cyclic subgroups of \( G \) are realized as normalizers.
\end{enumerate}

\end{proposition}

\begin{proof}
First we prove item $a)$. Let \( M \) be a Seifert $3$-manifold. Then by \cite[Lemma 3.2]{Sc83} its fundamental group fits into a short exact sequence
\begin{equation}\label{eq:sei}
1 \longrightarrow K \longrightarrow \pi_1(M) \longrightarrow \pi_1^{\mathrm{orb}}(B) \longrightarrow 1,
\end{equation}
where \( K \cong \mathbb{Z} \) is the infinite cyclic subgroup generated by a regular fiber, except in the case where \( M \) is finitely covered by \( S^3 \), in which case \( \pi_1(M) \) is finite and hence trivially balanced.

We analyze according to the geometry of the base orbifold \( B \):

\smallskip

\noindent
\emph{Case 1:} \( B \) is a bad orbifold or a spherical orbifold (i.e., modeled on \( S^2 \)).  
In this case, \( \pi_1^{\mathrm{orb}}(B) \) is finite, hence \( \pi_1(M) \) is virtually cyclic. In particular, it is balanced.

\smallskip

\noindent
\emph{Case 2:} \( B \) is modeled on the hyperbolic plane \( \mathbb{H}^2 \).  
Then \( \pi_1^{\mathrm{orb}}(B) \) is a Fuchsian group, which is balanced  by \cref{conditionC:hyperbolic:n-manifolds}. Applying \cref{sec:conditionC} to the extension \eqref{eq:sei}, we conclude that \( \pi_1(M) \) is balanced.

\smallskip

\noindent
\emph{Case 3:} \( B \) is modeled on the Euclidean plane.  
Then \( \pi_1^{\mathrm{orb}}(B) \) is virtually polycyclic, and thus balanced. Again, applying \cref{sec:conditionC} to \eqref{eq:sei}, it follows that \( \pi_1(M) \) is balanced.

We now prove item b). By \cref{conditionC:and:roots} and item~(a), it suffices to show that $\pi_1(M)$ has a normal finite-index subgroup satisfying uniqueness of roots. By \cite[Corollary~1.6]{BoyerRolfsenWiest2005}, the group $\pi_1(M)$ is virtually bi-orderable. Hence, there is a finite-index subgroup $H < \pi_1(M)$ that is bi-orderable and thus satisfies uniqueness of roots \cite{Rolfsen2014}. Consider the action of $\pi_1(M)$ on the set of left cosets $\pi_1(M)/H$; its kernel is a normal subgroup $K \trianglelefteq \pi_1(M)$ of finite index, contained in $H$. In particular, $K$ also satisfies uniqueness of roots. \end{proof}

\subsection{Balanced groups and mapping class groups}

\begin{definition}
Let \( M \) be a connected, compact, orientable topological manifold (of arbitrary dimension). The \emph{mapping class group} of \( M \) is defined as
\[
\mcg(M) := \pi_0(\mathrm{Homeo}^+(M)),
\]
where \( \mathrm{Homeo}^+(M) \) denotes the topological group of orientation-preserving homeomorphisms of \( M \).
\end{definition}

In \cite[Proposition 4.1]{JPT16}, it was shown that the mapping class group of a surface is balanced. In what follows, we show that this result also holds for mapping class groups of certain $3$-manifolds.

\begin{proposition}
Let \( S_g \) be a closed orientable surface of genus \( g \geq 1 \), and let \( X^k_g \) be a 3-manifold admitting the structure of an oriented circle bundle
\[
S^1 \hookrightarrow X^k_g \xrightarrow{p} S_g,
\]
with Euler number \( k \in \mathbb{Z} \). Then the mapping class group \( \mcg(X^k_g) \) is balanced.
\end{proposition}

\begin{proof}
According to \cite{lei:bena}, when \( (g,k)\neq (1,0) \), the mapping class group \( \mcg(X^k_g) \) fits into the short exact sequence:
\[
1 \to H^1(S_g;\mathbb{Z}) \to G \to \mathrm{Mod}(S_g) \to 1,
\]
where the kernel is a free abelian group of rank \( 2g \) and $G$ is a finite-index subgroup of $\mcg(X^{k}_{g})$. By \cite[Proposition 4.1]{JPT16} \( \mcg(S_g) \) is balanced, and since $H^1(S_g;\mathbb{Z})$ satisfies property (MAX), by \cref{propertyC:sec} we have that $G$ is balanced. Then Lemma \ref{conditionC:finite:extension} implies that $\mcg(X^{k}_{g})$ is balanced.
\end{proof}

\section{Virtually cyclic dimension of poly-surface groups}\label{virtually:cyclic:dimension}
In this section we obtain an upper bound for the virtually cyclic dimension of normally poly-surface groups and normally poly-free groups. Our first step is to show that these groups are balanced. In fact, we prove that every normally poly-hyperbolic group is balanced. Once this is established, we apply \cref{luck:condition C} to reduce the problem to the computation of the proper geometric dimension of the Weyl groups of infinite cyclic subgroups. 

Let \( \mathcal{P} \) be a class of groups (for example, hyperbolic groups, amenable groups, free groups, etc.). 
We say that a group \( G \) is \emph{poly-\(\mathcal{P}\)} if there exists a filtration
\[
1 = G_0 \lhd G_1 \lhd \cdots \lhd G_{n-1} \lhd G_n = G,
\]
such that each quotient \( G_{i+1}/G_i \), for \( 0 \leq i \leq n-1 \), belongs to the class \( \mathcal{P} \).
Furthermore, we say that \( G \) is \emph{normally poly-\(\mathcal{P}\)} if, in addition, each subgroup \( G_i \) is normal in \( G \). The minimal integer \( n \) for which such a filtration exists is called the \emph{length} of \( G \). 

The class of \emph{normally poly-\(\mathcal{P}\)} groups is under quotients.

\begin{lemma}\label{normally:quotient}
Let $G$ be a normally poly-$\mathcal{P}$ group of length $n$, with  filtration
\[
1 = G_0 \lhd G_1 \lhd G_2 \lhd \cdots \lhd G_n = G.
\]
Then, for any $1 \leq i \leq n$ the quotient group $G/G_i$ is a normally poly-$\mathcal{P}$ group of length at most $n-i$.
\end{lemma}

\begin{proof}
Consider the canonical projection $\pi \colon G \to G/G_i$. Since $G_i \triangleleft G$.
For  each $j \ge i$, the subgroup $G_j$ is normal in $G$, hence by the correspondence theorem its image
\[
\pi(G_j) = G_j/G_i
\]
is a normal subgroup of $G/G_i$. Consequently, we obtain a normal filtration of length $n-i$
\[
1 = G_i/G_i \lhd G_{i+1}/G_i \lhd \cdots \lhd G_n/G_i = G/G_i.
\]

Moreover, for each $j \ge 2$, the successive quotients satisfy
\[
\frac{G_j/G_i}{G_{j-1}/G_i} \cong \frac{G_j}{G_{j-1}},
\]
where the isomorphism is induced by the projection $\pi$ and follows from the Third Isomorphism Theorem.
By hypothesis, each group $G_j/G_{j-1}$ belongs to $\mathcal{P}$. Therefore, the same holds for each quotient
\[
(G_j/G_i)/(G_{j-1}/G_i).
\]
 Hence, $G/G_i$ is a normally poly-surface group of length at most $n-i$.
\end{proof}

\begin{lemma}\label{normally:surface:groups:closed:under:subgroups}
Let $G$ be a normally poly-surface group of length $n$, with filtration
\[
1 = G_0 \lhd G_1 \lhd \cdots \lhd G_n = G
\]
such that $G_i \lhd G$ for every $i$ and
\[
G_i/G_{i-1} \cong \pi_1(S_i)
\]
for some connected surface $S_i$.
Then every subgroup $H \le G$ is normally poly-surface of length at most $n$.
\end{lemma}

\begin{proof}
Let $H \le G$ and set $H_i := G_i \cap H$ for $i=0,\dots,n$. Then
\[
1 = H_0 \le H_1 \le \cdots \le H_n = H.
\]

We first show that $H_i \triangleleft H$ for every $i$.  
Let $h \in H$ and $x \in H_i$. Since $x \in G_i$ and $G_i \lhd G$, we have
$hxh^{-1} \in G_i$. Moreover, as $H$ is a subgroup, $hxh^{-1} \in H$. Hence
$hxh^{-1} \in G_i \cap H = H_i$, proving $H_i \lhd H$.

For each $i \ge 1$, because $G_{i-1} \lhd G_i$, the second isomorphism theorem yields
\[
\frac{H_i}{H_{i-1}}
=
\frac{G_i \cap H}{G_{i-1} \cap H}
\;\cong\;
\frac{(G_i \cap H)G_{i-1}}{G_{i-1}}.
\]
Since $(G_i \cap H)G_{i-1} \le G_i$, we obtain an embedding
\[
\frac{H_i}{H_{i-1}}
\;\hookrightarrow\;
\frac{G_i}{G_{i-1}}
\;\cong\;
\pi_1(S_i).
\]

Thus $H_i/H_{i-1}$ is isomorphic to a subgroup of $\pi_1(S_i)$.  Since surface groups are closed under subgroups we have that $H_i/H_{i-1}$ is a surface group. 

We conclude that $H$ admits a filtration
\[
1 = H_0 \lhd H_1 \lhd \cdots \lhd H_n = H
\]
in which each term is normal in $H$ and each successive quotient is a surface group. Hence $H$ is normally poly-surface.
\end{proof}

 


\begin{proposition}\label{hyperbolic:groups:propertyC}
 If $G$ is a virtually normally poly-hyperbolic group of length $n$, then $G$ is balanced.   
\end{proposition}
\begin{proof}
 By Lemma \ref{conditionC:finite:extension}, we may suppose that $G$ is a normally poly-hyperbolic group of length $n$. The proof is by induction on $n$. 
 
 If $n=1$, then $G$ is a hyperbolic group, and by Example \ref{example:condition C} (1) we have that $G$ is balanced. Suppose the result holds for normally poly-hyperbolic groups of length $n\leq k-1$.

Let $G$ be a normally poly-hyperbolic group  of length $k$. Then there is a filtration  of $G$ by subgroups   $1=G_0\lhd G_1\lhd \cdots \lhd G_{k-1}\lhd G_k=G$  such that for each $i$  $G_i\lhd G$ and the quotient $G_{i+1}/G_i$ is a hyperbolic group. We consider the following short exact sequence 
$$1\to G_{1}\to G\xrightarrow[]{p} G/G_{1}\to 1.$$ 
By \cref{normally:quotient}, $G/G_1$ is a  normally poly-hyperbolic group of length at most $k-1$ and $G_1$ is a hyperbolic group. This implies that $G_1$ has property (MAX) due to Remark \ref{remark:examples MAX}(1), and $G/G_1$ is balanced due to the induction hypothesis. Then, the claim follows from \cref{propertyC:sec}.
\end{proof}

\begin{remark}\label{rem:poly-free and poly-surface are balanced}
The class of normally poly-hyperbolic groups includes, in particular, all normally poly-finitely-generated-free groups (poly-f.g.-free groups). Moreover, exactly the same argument used in the proof above applies directly to all normally poly-free groups (without the finite-generation condition), and also to all normally poly-surface groups.
\end{remark}

\subsection*{Virtually cyclic dimension of poly-surface groups }

In this subsection we compute some upper bounds for some geometric dimensions for poly-surfaces groups and poly-free groups. In particular we prove Theorems \ref{poli:surface:th1} and \ref{poli:free:virtually:cy} (Theorems \ref{thm:poly-surface} and \ref{thm:poly-free} below, respectively).

Let $G$ be a discrete group and a family $\calF$ of subgroups of $G$. We define the geometric dimension $\gd_{\calF}(G)$ as the minimum integer $n$ such that there is a model $X$ for the classifying space $E_{\calF}G$ of dimension $n$. In the literature when the family $\calF$ is the trivial family,  the finite subgroups and the cyclic subgroups, the geometric dimension is denoted by $\gd$, $\gdfin(G)$, $\gdvc(G)$ respectively.

\begin{remark}
    Note that if $\mathcal{TF}$ is the collection of torsion-free groups, then any poly-$\mathcal{TF}$ group is also torsion-free. Indeed, let $G$ be a poly-$\mathcal{TF}$ group with filtration $1 = G_0 \lhd G_1 \lhd \cdots \lhd G_{n-1} \lhd G_n = G$; if there exists a finite-order element $g \in G= G_n$, since $G_n/G_{n-1}$ is torsion free, then $g$ has to be an element of $G_{n-1}$. Repeating this argument $n-1$ times, we get that $g \in G_0$ which means it is trivial.
\end{remark}

The following is a mild generalization of  \cite[Theorem 1.1]{MR4817678}.

\begin{lemma}
Let $\mathcal{TF}$ be the collection of torsion-free groups, and let $G$ be a poly-$\mathcal{TF}$ group of length $n$, with a filtration
\[
1 = G_0 \lhd G_1 \lhd \cdots \lhd G_{n-1} \lhd G_n = G.
\]
Then
\[
\gd(G) \leq \sum_{i=1}^n \gd(G_i / G_{i-1}).
\]
\end{lemma}

A simple fact of poly-$\mathcal{TF}$ groups that we use thoroughly in this section (since poly-free groups and poly-surface groups are poly-$\mathcal{TF}$), is that a poly-$\mathcal{TF}$ group is torsion-free. Indeed, if $G$ admits a filtration $1 = G_0 \lhd G_1 \lhd \cdots \lhd G_{n-1} \lhd G_n = G$, and $g \in G$ has finite order, then $g \in G_{n-1}$ since $G/G_{n-1}$ is torsion-free. After repeating this argument $n$ times we have that $g \in G_0 = \{e\}$.

\begin{proof}
We proceed by induction on the length \( n \) of the filtration.

If \( n = 1 \), then \( G = G_1 \) and the claim is clear.

Assume the statement holds for all poly-$\mathcal{TF}$ groups of length less than \( n \). Let \( G \) be a poly-$\mathcal{TF}$ group of length \( n \) with a filtration
\[
1 = G_0 \lhd G_1 \lhd \cdots \lhd G_{n-1} \lhd G_n = G.
\]
Consider the short exact sequence
\[
1 \to G_{n-1} \to G \to G_n / G_{n-1} \to 1.
\]
Since $G$ and \( G_n / G_{n-1} \) are torsion-free, then $\gd(G) = \gdfin(G)$ and $\gd(G_n/G_{n-1}) = \gdfin(G_n/G_{n-1})$, and it follows from \cite[Theorem 5.15]{Lu05} that
\[
\gd(G) \leq \gd(G_{n-1}) + \gd(G_n / G_{n-1}).
\]
By the induction hypothesis,
\[
\gd(G_{n-1}) \leq \sum_{i=1}^{n-1} \gd(G_i / G_{i-1}).
\]
Therefore,
\[
\gd(G) \leq \sum_{i=1}^{n} \gd(G_i / G_{i-1}),
\]
as claimed.
\end{proof}

Using that any poly-free group and any poly-surface group is in particular poly-$\mathcal{TF}$, we obtain the following consequences.

\begin{corollary}\label{gd:poli-free}
Let \( G \) be a poly-free group of length \( n \), with a filtration
\[
1 = G_0 \lhd G_1 \lhd \cdots \lhd G_{n-1} \lhd G_n = G.
\]
Then
\[
\gd(G) \leq n.
\] 
\end{corollary}

\begin{corollary}\label{gd:poly-surface}
Let \( G \) be a poly-surface group of length \( n \), with a filtration
\[
1 = G_0 \lhd G_1 \lhd \cdots \lhd G_{n-1} \lhd G_n = G.
\]
Then
\[
\gd(G) \leq 2n.
\] 
\end{corollary}

Now, since poly-$\mathcal{TF}$ groups are torsion-free, we have that $\gd = \gdfin$. So, henceforth we focus only on finding an upper bound to $\gdvc$ of these groups.

\begin{lemma}\label{exten:gd:inequality}\cite[Theorem 5.16]{Lu05}
Let 
\[
1 \to K \to G \to H \to 1
\]
be a short exact sequence of groups. Suppose that for any group \( M \) which contains a finite-index subgroup isomorphic to \( K \), we have \( \gdfin(M) = \gdfin(K) \). Then the following inequality holds
\[
\gdfin(G) \leq \gdfin(K) + \gdfin(H).
\]
\end{lemma}
Motivated by Lemma~\ref{exten:gd:inequality}, we introduce the following notion.

\begin{definition}
A group \(G\) is said to be \emph{geometrically rigid} if for every finite extension
\(K\) of \(G\) one has
\[
\gdfin(G)=\gdfin(K).
\]
\end{definition}

\medskip

\begin{example}\label{lem:rigid-classes}
    Several natural classes of groups are geometrically rigid.
    \begin{enumerate}
        \item Free groups (this goes back to work of Dunwoody in the finitely generated case; see \cite[Corollary~2.1]{MR4817678} for the general case).
        \item Polycyclic groups (see \cite{LW12}).
        \item Fundamental groups of closed surfaces (see \cite[Lemma 4.4]{Gui2010}
    \end{enumerate}
\end{example}


\begin{lemma}\cite[Lemma 4.4]{MR2545612}\label{luck:condition C}
 Let $n$ be an integer. Suppose that $G$ is a balanced group. Suppose that $\gdfin(G)\leq n$ and for every infinite cyclic subgroup $H$ of $G$ we have $\gdfin(W_G(H))\leq n$. Then $\gdvc(G)\leq n+1$.  
\end{lemma}
\begin{theorem}\label{thm:poly-free}
Let $G$ be a normally poly-free group of length $n$, then $\gdvc(G)\leq n+1$ 
\end{theorem}

\begin{proof}
By \cref{hyperbolic:groups:propertyC} and \cref{rem:poly-free and poly-surface are balanced} $G$ is balanced and by \cref{gd:poli-free} $\gd(G)\leq n$, then by \cref{luck:condition C}  it is enough to prove that for every infinite cyclic subgroup $H$ of $G$ we have $\gdfin(W_G(H))\leq n-1$. We prove this by induction on $n$.
If $n=1$, then $G$ is a free group, and it is well-known that every infinite virtually cyclic subgroup has the normalizer $N_G(H)$ virtually cyclic it follows that $W_G(H)$ is finite and in particular $\gdfin(W_G(H))=0$. 

We now suppose the claim is true for every $k<n$, i.e., we suppose that for every normally poly-free group $G$  of length $k<n$ and for every infinite virtually cyclic subgroup $H<G$ we have $\gdfin(W_G(H))\leq k-1$. We prove the claim for $n=k$.
 Let $G$ be a normally poly-free group  of length $k$. Then there is a filtration  of $G$ by subgroups   $1=G_0\lhd G_1\lhd \cdots \lhd G_{k-1}\lhd G_k=G$  such that $G_i\lhd G$, and the quotient $G_{i+1}/G_i$ is a  free group. We consider the following short exact sequence 
$$1\to G_{1}\to G\xrightarrow[]{p} G/G_{1}\to 1.$$ 

and let  $H$ be an infinite cyclic subgroup of $G$, we can restrict the short exact sequence to 

$$1\to G_{1}\cap N_G(H) \to N_G(H)\xrightarrow[]{p} Q\to 1.$$ 

We have two cases $p(H)=1$ or $p(H)\neq1$. 

In the first case we have the following short exact sequence 

$$1\to  N_{G_1}(H)/H \to W_G(H)\xrightarrow[]{p} Q\to 1.$$

Note that $ N_{G_1}(H)/H$ is a finite subgroup, then by \cref{exten:gd:inequality} we have \[\underline{\gd}(W_G(H))\leq 0 + \underline{\gd}(Q) \leq n-1\]
the last inequality is because $Q$ is a normally poly-free group of length at most $\leq n-1$ (\cref{normally:quotient}).

Now suppose that $p(H)\neq1$. In this case we have the following short exact sequence 
$$1\to G_{1}\cap N_G(H) \to W_G(H)\xrightarrow[]{p} Q/p(H)\to 1.$$

Note that $G_{1}\cap N_G(H)$ is a free group, and by \cref{lem:rigid-classes}(1) any non-trivial virtually free group has $\gdfin=1$, then it follows by \cref{exten:gd:inequality} that 

\[\underline{\gd}(W_G(H))\leq 1 + \underline{\gd}(Q/p(H))\leq 1+ \underline{\gd}(N_{G/G_1}(p(H))/p(H))\leq n-1.\]

\end{proof}


\begin{theorem}\label{thm:poly-surface}
Let $G$ be a normally poly-surface group of length $n$, then $\gdvc(G)\leq 2n$ 
\end{theorem}
\begin{proof}
By \cref{hyperbolic:groups:propertyC} and \cref{rem:poly-free and poly-surface are balanced} $G$ is a balanced group and by \cref{gd:poly-surface} $\gd(G)\leq 2n$. Then, using \cref{luck:condition C}  it is enough to prove that for every infinite cyclic subgroup $H$ of $G$ we have $\gdfin(W_G(H))\leq 2n-1$. We prove this by induction on $n$.

If $n=1$, then $G$ is a surface group, since $G$ is a torsion-free hyperbolic group, a free group,  $\Z^2$ or $\Z\rtimes \Z$, it follows that every infinite virtually cyclic subgroup has the normalizer $N_G(H)$ infinite cyclic, $\Z^2$ or $\Z\rtimes \Z$. Then $W_G(H)$ is finite or $\Z$, in particular $\gdfin(W_G(H))$ is $0$ or $1$, and in any case $\gdfin(W_G(H))\leq 1$. 

We now suppose the claim is true for every $k<n$, i.e., we suppose that for every normally poly-surface group $G$  of length $k<n$ and for every infinite cyclic subgroup $H<G$ we have $\gdfin(W_G(H))\leq 2k-1$. We prove the claim for $n=k$.
 Let $G$ be a normally poly-surface group of length $k$. Then there is a filtration  of $G$ by subgroups   $1=G_0\lhd G_1\lhd \cdots \lhd G_{k-1}\lhd G_k=G$  such that for all $i$, $G_i\lhd G$ and the quotient $G_{i+1}/G_i$ is a  surface group. We consider the following short exact sequence 
$$1\to G_{1}\to G\xrightarrow[]{p} G/G_{1}\to 1,$$ 
and let  $H$ be an infinite cyclic subgroup of $G$ generated by $h$. We can restrict the short exact sequence to 

$$1\to G_{1}\cap N_G(H) \to N_G(H)\xrightarrow[]{p} Q\to 1.$$ 

We divide the proof into two cases: $p(H)=1$ or $p(H)\neq1$. Moreover, since $G/G_1$ is a normally poly-surface group, it is torsion free, which implies that $Q$ is torsion-free. Hence, $p(H) = 1$ if and only if there exists $n \in \Z \setminus \{0\}$ such that $p(h^n) = 1$.

\textit{Case 1,} $p(H) = 1$: In this case, $H$ is a subgroup of $G_1 \cap N_G(H)$ and we can obtain the following short exact sequence 
$$1\to  N_{G_1}(H)/H \to W_G(H)\xrightarrow[]{p} Q\to 1.$$

Note that $G_1$ is a surface group, which implies that $ N_{G_1}(H)/H$ is a finite  or infinite cyclic subgroup, and in any case they are geometrically rigid and have $\gdfin$ at most 1. Then, by \cref{exten:gd:inequality} we have \[\underline{\gd}(W_G(H))\leq \underline{\gd}(N_{G_1}(H)/H) + \underline{\gd}(Q) \leq 1+2(k-1)=2k-1\]
the last inequality is because $Q$ is a normally poly-surface group of length at most $k-1$ (by \cref{normally:quotient} we have that $G/G_1$ is a poly-surface group of length at most $k-1$ and $Q$ is a subgroup of $G/G_1$, now it follows from \cref{normally:surface:groups:closed:under:subgroups}).

\textit{Case 2,} $p(H) \neq 1$: As mentioned before, this case is equivalent to having that for all $n \in \Z \setminus \{0\}$ we have that $p(h^n) \neq 1$. Then, $H \cap (G_1 \cap N_G(H)) = 1$, and we have the following short exact sequence 
$$1\to G_{1}\cap N_G(H) \to W_G(H)\xrightarrow[]{p} Q/p(H)\to 1.$$

Note that $G_{1}\cap N_G(H)$ is, in all cases, either a free group, a poly-$\mathbb Z$
group of Hirsch length at most $2$, or a surface group. Thus, $G_1 \cap N_G(H)$ is geometrically rigid and has $\gdfin$ at most 2. It follows from \cref{lem:rigid-classes} together with
Lemma~\ref{exten:gd:inequality} that

\[
\begin{array}{rcl}
    \underline{\gd}(W_G(H)) & \leq & \underline{\gd}(G_{1} \cap N_{G}(H)) + \underline{\gd}(Q/p(H)) \\
     & \leq & 2 + \underline{\gd}(Q/p(H)) \\
     & \leq & 2+ \underline{\gd}(N_{G/G_1}(p(H))/p(H)) \\
     & \leq & 2 + 2(k-1)-1\\
     & = & 2k-1
\end{array}
\]

\end{proof}

\bibliographystyle{alpha} 
\bibliography{mybib}
\end{document}

%% file: Preamble.tex
\usepackage{graphicx}
\usepackage{pinlabel} 
\usepackage{amsmath}
\usepackage{amsfonts}
\usepackage{amssymb}
\usepackage{mathtools}
\usepackage[all,cmtip]{xy}
\usepackage{amsthm}
\usepackage{tikz-cd}
\usepackage{comment}
\usepackage{enumerate}
\usepackage{url}
\usepackage{epsfig}
\usepackage{hyperref}
\usepackage[utf8]{inputenc}
\usepackage[capitalise]{cleveref}
\usepackage{geometry}
\makeatletter
\providecommand\@dotsep{5}
\def\listtodoname{List of Todos}
\def\listoftodos{\@starttoc{tdo}\listtodoname}
\makeatother

\newtheorem{theorem}{Theorem}[section]
\newtheorem{proposition}[theorem]{Proposition}
\newtheorem{corollary}[theorem]{Corollary}
\newtheorem{lemma}[theorem]{Lemma}

\newtheorem{remark}[theorem]{Remark}

\newcommand{\mycomment}[1]{}

  \theoremstyle{definition}
\newtheorem{definition}[theorem]{Definition}
\newtheorem{example}[theorem]{Example}

\newcommand{\Z}{\mathbb{Z}}

\newcommand{\calF}{{\mathcal F}}

\newcommand{\vcyc}{V\text{\tiny{\textit{CYC}}}}

\newcommand{\nbeq}{\begin{equation}}
\newcommand{\neeq}{\end{equation}}
\newcommand{\beq}{\begin{equation*}}
\newcommand{\eeq}{\end{equation*}}

\DeclareMathOperator{\gd}{gd}
\DeclareMathOperator{\gdfin}{\underline{gd}}
\DeclareMathOperator{\gdvc}{\underline{\underline{gd}}}

\DeclareMathOperator{\mcg}{Mod}

\DeclareMathOperator{\Aut}{Aut}

\DeclareMathOperator{\cat}{CAT}

